\documentclass[12pt]{amsart}
\usepackage{amsfonts,amssymb,amscd,amsmath,enumerate,verbatim,calc}
\usepackage{graphicx}

\newtheorem{Theorem}{Theorem}[section]
\newtheorem{Lemma}[Theorem]{Lemma}
\newtheorem{Corollary}[Theorem]{Corollary}
\newtheorem{Proposition}[Theorem]{Proposition}
\newtheorem{Remark}[Theorem]{Remark}

\begin{document}
\title{On Fixing number of Functigraphs }
\author{Muhammad Fazil, Imran Javaid, Muhammad Murtaza}
\keywords{fixing number, functigraph.\\
2010 {\it Mathematics Subject Classification.} 05C25\\
$^*$ Corresponding author: mfazil@bzu.edu.pk}

\address{Centre for advanced studies in Pure and Applied Mathematics,
Bahauddin Zakariya University Multan, Pakistan
\newline Email:
mfazil@bzu.edu.pk, imran.javaid@bzu.edu.pk, mahru830@gmail.com}

\date{}
\maketitle
\begin{abstract} The fixing number of a graph $G$ is the order of
the smallest subset $S$ of its vertex set $V(G)$ such that
stabilizer of $S$ in $G$, $\Gamma_{S}(G)$ is trivial. Let $G_{1}$
and $G_{2}$ be disjoint copies of a graph $G$, and let
$g:V(G_{1})\rightarrow V(G_{2})$ be a function. A functigraph
$F_{G}$ consists of the vertex set $V(G_{1})\cup V(G_{2})$ and the
edge set $E(G_{1})\cup E(G_{2})\cup \{uv:v=g(u)\}$. In this paper,
we study the behavior of the fixing number in passing from $G$ to
$F_{G}$ and find its sharp lower and upper bounds. We also study the
fixing number of functigraphs of some well known families of graphs
like complete graphs, trees and join graphs.
\end{abstract}

\section{Introduction}
Let $G$ be a connected graph with vertex set $V(G)$ and edge set
$E(G)$. The number of vertices and edges is called the \emph{order}
and the \emph{size} of $G$, respectively. Two vertices $u$ and $v$
of $G$ are called \emph{connected} if there is a path between $u$
and $v$ in $G$. The \emph{distance} between two vertices $u$ and $v$
in $G$, denoted by $d(u,v)$, is the length of a shortest $u-v$ path
if $u$ and $v$ are connected. The \emph{degree} of the vertex $v$ in
$G$, denoted by $deg_{G}(v)$, is the number of edges to which $v$
belongs. The \emph{open neighborhood} of the vertex $u$ of $G$ is
$N(u)=\{v\in V(G):uv\in E(G)\}$ and the \emph{closed neighborhood}
of $u$ is $N(u)\cup \{u\}$. Two vertices $u,v$ are \emph{adjacent
twins} if $N[u]=N[v]$ and \emph{non adjacent twins} if $N(u)=N(v)$.
If $u,v$ are adjacent or non adjacent twins, then $u,v$ are
\emph{twins}. A set of vertices is called \emph{twin-set} if every
of its two vertices are twins.

An \textit{automorphism} $\alpha$ of $G$, $\alpha: V(G)\rightarrow
V(G),$ is a bijective mapping such that $\alpha(u)\alpha(v)\in E(G)
$ if and only if $uv \in E(G).$ Thus, each automorphism $\alpha$ of
$G$ is a \emph{permutation} of the vertex set $V(G)$ which preserves
adjacencies and non-adjacencies. The \textit{automorphism group} of
a graph $G$, denoted by $\Gamma(G)$, is the set of all automorphisms
of a graph $G$. The \emph{stabilizer} of $v$, denoted by
$\Gamma_{v}(G)$, is the set $\{\alpha\in \Gamma(G) : v=
\alpha(v)\}$. The stabilizer of a set of vertices $S\subseteq V (G)$
is $\Gamma_{S}(G) = \{\alpha\in \Gamma(G) : v= \alpha(v) \,\ \forall
\,\ v\in S\}$. Note that $\Gamma_{S}(G)=\bigcap_{v\in S}
\Gamma_{v}(G)$. The \emph{orbit} of a vertex $v$, denoted by
$\theta(v)$, is the set $\{u \in V (G) : u= \alpha(v)\,\ \mbox{for
some}\,\ \alpha\in \Gamma(G)\}$. Two vertices $u$ and $v$ are
\emph{similar}, denoted by $u\sim v$ if they belong to the same
orbit.

A vertex $v$ is \emph{fixed} by a group element $\alpha\in
\Gamma(G)$ if $\alpha\in \Gamma_{v}(G)$. A set of vertices
$S\subseteq V(G)$ is a \emph{fixing set} of $G$ if $\Gamma_{S}(G)$
is trivial. In this case, we say that $S$ fixes $G$. The
\emph{fixing number} of a graph $G$, denoted by $fix(G)$, is the
cardinality of a smallest fixing set of $G$ \cite{erw}. The graphs
with $fix(G) = 0$ are called the \emph{rigid graphs} \cite{alb},
which have trivial automorphism group. Every graph $G$ has a fixing
set. Trivially, the set of vertices of $G$ itself is a fixing set.
It is also clear that any set containing all but one vertex is a
fixing set. Thus, for a graph $G$ on $n$ vertices $0\leq fix(G)\leq
n-1$.

The fixing number of graph $G$ was first defined by Erwin and Harary
in 2006 \cite{erw}. Boutin independently, did her research on the
fixing number and name this parameter, the \emph{determining
number}. A considerable literature has been developed in this field
(see \cite{cac, gib, har}).

Let $G_{1}$ and $G_{2}$ be disjoint copies of a connected graph $G$,
and let $g:V(G_{1})\rightarrow V(G_{2})$ be a function. A
\emph{functigraph} $F_{G}$ of graph $G$ consists of the vertex set
$V(G_{1})\cup V(G_{2})$ and the edge set $E(G_{1})\cup E(G_{2})\cup
\{uv:v=g(u)\}$.

The idea of permutation graph was introduced by Chartrand and Harary
\cite{char} for the first time. Dorfler \cite{dor}, introduced a
mapping graph which consists of two disjoint identical copies of
graph where the edges between the two vertex sets are specified by a
function. The mapping graph was rediscovered and studied by Chen et
al. \cite{yi}, where it was called the functigraph. Unless otherwise
specified, all the graphs $G$ considered in this paper are simple,
non-trivial and connected. Throughout the paper, we will denote the
functigraph of $G$ by $F_{G}$, $V(G_{1})=A$, $V(G_{2})=B$,
$g(V(G_{1}))= I$, $|g(V(G_{1}))|= |I|= s$ and the minimum fixing set
of $F_{G}$ by $S^{\ast}$.

This paper organized as follows. Section 2 provides the study of
fixing number of functigraphs. We give sharp lower and upper bounds
for fixing number of functigraph. This section also establishes the
connections between the fixing number of graphs and their
corresponding functigraphs in the form of realizable results.
Section 3 provides the fixing number of functigraphs of some well
known families of graphs likes complete graphs, trees and join
graphs. Some useful results related to these families are also part
of this section.

\section{Some basic results and bounds}
We recall some elementary results about the fixing number which are
useful for onward discussion.

%

\begin{Proposition}\label{fazf}\cite{faz} Suppose that $u$, $v$ are twins in a connected
graph $G$ and $S$ is a fixing set of $G$, then either $u$ or $v$ is
in $S$. Moreover, if $u\in S$ and $v\not\in S$, then
$(S-\{u\})\cup\{v\}$ is a fixing set of $G$.
\end{Proposition}
\begin{Proposition}\label{rem3.3}\cite{faz}
Let $U$ be a twin-set of order $m\geq 2$ in a connected graph $G$,
then every fixing set $S$ of $G$ contains at least $m-1$ vertices of
$U$.
\end{Proposition}

An ordered set $W = \{w_{1},w_{2},...,w_{k}\}\subseteq V(G)$ is
called a \emph{resolving set} for $G$ if for every two distinct
vertices $u$ and $v$ of $G$, $(d(u,w_1),\ldots,d(u,w_k))\neq
(d(v,w_1),\ldots,d(v,w_k))$. The \emph{metric dimension}, denoted by
$\beta(G)$, is the cardinality of a minimum resolving set of $G$
\cite{har1}.

\begin{Proposition}\label{f3} \cite{bou, erw} If $S\subseteq V(G)$
is a resolving set of $G$, then $S$ is a fixing set of G. In
particular, $fix(G)\leq \beta(G)$.
\end{Proposition}

%

\begin{Theorem} \label{f4}\cite{kang} Let $G$ be a connected graph
of order $n\geq 3$, and let $g:A\rightarrow B$ be a function, then
$2\leq \beta(F_{G})\leq 2n-3$. Both bounds are sharp.
\end{Theorem}

Next, the sharp lower and upper bounds on the fixing number of
functigraphs are given in the following result:

\begin{Proposition} \label{f5}Let $G$ be a connected graph of order $n\geq 3$, and
let $g:A\rightarrow B$ be a function, then $0\leq fix(F_{G})\leq
2n-3$. Both bounds are sharp.
\end{Proposition}
\begin{proof}Obviously, $0\leq
fix(F_{G})$ by definition. The upper bound follows from Proposition
\ref{f3} and Theorem \ref{f4}. Hence $0\leq fix(F_{G})\leq 2n-3$.
For the sharpness of the lower bound, take $G=P_{3}$ and
$g:A\rightarrow B$, be a function such that $v_{1}= g(u_{i}), i=1,2$
and $v_{3}= g(u_{3})$. For the sharpness of the upper bound, take
$G=K_{n}$, the complete graph of order $n\geq 3$, and let
$g:A\rightarrow B$ be defined by $v_{1}= g(u_{i})$ for each $i,$
where $1\leq i\leq n$. Hence $fix(F_{G})=2n-3$ and the proof is
complete.
\end{proof}

A connected graph $G$ is called \emph{symmetric} if $fix(G)\neq 0.$
By using the Proposition \ref{f5}, we have the following result:
\begin{Proposition} Let $G$ be a symmetric connected graph,
then $1\leq fix(G)+ fix(F_{G})\leq 3n-4.$ Both bounds are sharp.
\end{Proposition}

\begin{figure}[h]
        \centerline
        {\includegraphics[width=4cm]{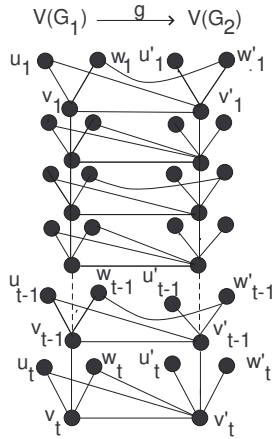}}
        \label{Fix3}
        \caption{The graph with $fix(G)= t = fix(F_{G}).$}\label{f1}
\end{figure}

\begin{Lemma}\label{fff} For any integer $t\geq 2$,
there exists a connected graph $G$ and a function $g$ such that
$fix(G)=t=fix(F_{G})$.
\end{Lemma}
\begin{proof}Construct the graph $G$ as follows: Let $P_{t}:
v_{1}v_{2}...v_{t}$ be a path. Join two pendant vertices $u_{i},
w_{i}$ with each $v_{i}$, where $1\leq i \leq t$. This completes
construction of $G$. Note that one of the vertex from each pair of
pendant vertices belongs to a minimum fixing set of $G$ and hence
$fix(G)=t$. Now, we label the corresponding vertices of $B$ as
$v_{i}', u_{i}', w_{i}'$ for all $i$, where $1\leq i \leq t$ and
construct a functigraph $F_{G}$ as follows: If $t$ is even, then
define $g:A\rightarrow B$ as $v_{i}'= g(v_{i})$, for all $i$, where
$1\leq i \leq t;$ $v_{i}'= g(u_{i}), w_{i}'= g(w_{i})$, for all $i=
2k+1$, where $0\leq k \leq \frac{t}{2}-1;$ and $v_{i}'= g(u_{i})=
g(w_{i})$, for all $i= 2k$, where $1\leq k \leq \frac{t}{2}$ as
shown in the Figure \ref{f1}. Now, consider the set $S^{\ast}=
\{u_{i}, u_{i}';\,\ i=2k,\,\ \mbox{where}\,\ 1\leq k \leq
\frac{t}{2}\}$. Note that $\Gamma_{S^{\ast}}(F_{G})$ is trivial and
hence $S^{\ast}$ is a fixing set of $F_{G}.$ Thus, $fix(F_{G})\leq
t.$ Moreover, $N(u_{i})= N(w_{i})$ and $N(u_{i}')= N(w_{i}')$, for
all $i= 2k$, where $1\leq k \leq \frac{t}{2}$. Thus, we have twin
sets $\{u_i,w_i\}$, $\{u'_i,w'_i\}$ for all $i=2k$, where $1\le k
\le \frac {t}{2}$. By Proposition \ref{rem3.3}, at least one element
from these $t$ twin sets must belongs to every fixing set of
$F_{G}$. This implies that $fix(F_{G})\geq t$. Hence, $fix(F_{G})=
t$. If $t$ is odd, then we define $g:A\rightarrow B$ by $v_{i}'=
g(v_{i})$, for all $i$, where $1\leq i \leq t;$ $v_{i}'= g(u_{i}),
w_{i}'= g(w_{i})$, for all $i= 2k+1$, $0\leq k \leq
\lfloor\frac{t}{2}\rfloor-1;$ $v_{i}'= g(u_{i})= g(w_{i})$, for all
$i= 2k$, $1\leq k \leq \lfloor\frac{t}{2}\rfloor;$ and $v_{t-2}'=
g(u_{t}),$ $w_{t-2}'= g(w_{t}).$ Use same steps as for case when $t$
is even and choosing $S^{\ast}= \{u_{i}, u_{i}';\,\ i=2k,\,\
\mbox{where}\,\ 1\leq k \leq \lfloor\frac{t}{2}\rfloor\}$, we note
that $fix(F_{G})= t$. Hence $fix(G)=t=fix(F_{G}).$
\end{proof}

Let us now discuss a functigraph of graph $G$ as described in proof
of Lemma \ref{fff} and function $g:A\rightarrow B$ defined as: If
$t$ is even, then $v_{i}'= g(v_{i})$, for all $1\leq i \leq t;$
$v_{t-1}'= g(u_{t-1})= g(w_{t-1});$ $v_{t-3}'= g(u_{t})$, $w_{t-3}'=
g(w_{t})$; $v_{i}'= g(u_{i}), w_{i}'= g(w_{i})$, for all $i= 2k+1$,
$0\leq k \leq \frac{t}{2}-2;$ $v_{i}'= g(u_{i})= g(w_{i})$, for all
$i= 2k$, $1\leq k \leq \frac{t}{2}-1$. Now, if $t$ is odd and
$g:A\rightarrow B$ is defined by $v_{i}'= g(v_{i})$, for all $1\leq
i \leq t;$ $v_{t}'= g(u_{t})= g(w_{t});$ $v_{i}'= g(u_{i}), w_{i}'=
g(w_{i})$, for all $i= 2k+1$, $0\leq k \leq
\lfloor\frac{t}{2}\rfloor-1;$ $v_{i}'= g(u_{i})= g(w_{i})$, for all
$i= 2k$, $1\leq k \leq \lfloor\frac{t}{2}\rfloor.$ From this
construction and using same arguments as in proof of Lemma
\ref{fff}, we conclude that $fix(G)=t$ and $fix(F_G)= t+1.$ Hence,
we have the following result:
\begin{Lemma}\label{fffw} For any two integers $t_{1}\geq 2$ and $t_{2}=t_{1}+1$,
there exists a connected graph $G$ and a function $g$ such that
$fix(G)=t_{1}, fix(F_{G})= t_{2}$.
\end{Lemma}
\begin{Remark} Let $t_{1}, t_{2}\geq 2$ be any two integers,
then by definition of functigraph, it is not necessary that there
always exists a connected graph $G$ such that $fix(G)=t_{1},
fix(F_{G})= t_{2}$.
\end{Remark}

Consider an integer $t\geq 2$. For $t= 2$, we take $G=P_{3}$ and its
functigraph $F_G$, where function $g: A\rightarrow B$ is defined as:
$v_{i}'= g(v_{i})$ for all $1\leq i \leq 3$. For $t> 2$, we take
graph $G$ same as in proof of Lemma \ref{fff} by taking path of
order $t-1$ and its functigraph $F_{G}$, where $g: A\rightarrow B$
is defined as: $v_{i}'= g(v_{i})$, for all $1\leq i \leq t-1;$
$v_{t-2}'= g(u_{t-1}), w_{t-2}'= g(w_{t-1});$ $v_{i}'= g(u_{i}),
w_{i}'= g(w_{i}), 1\leq i \leq t-2.$ From this construction, we
distinguish the following result which shows that
$fix(G)+fix(F_{G})$ can be arbitrary large:
\begin{Lemma}\label{fffw1}
For any integer $t\geq 2$, there exists a connected graph $G$ and a
function $g$ such that $fix(G)+fix(F_{G})=t$.
\end{Lemma}

Consider $t\geq 2$. We take graph $G$ by taking path of order $t+1$
and its functigraph $F_G$ as constructed in Lemma \ref{fffw1}, we
distinguish the following result which shows that
$fix(G)-fix(F_{G})$ can be arbitrary large:
\begin{Lemma}\label{fffw11}
For any integer $t\geq 2$, there exists a connected graph $G$ and a
function $g$ such that $fix(G)-fix(F_{G})=t$.
\end{Lemma}
For $t\geq 2$, we take graph $G$ same as in proof of Lemma \ref{fff}
by taking path of order $t$ and its functigraph $F_{G}$, where $g:
A\rightarrow B$ is defined as: $v_{i}'= g(v_{i})= g(u_{i})=
g(w_{i})$, for all $1\leq i \leq t.$ From this type of construction,
we distinguish the following result which shows that
$fix(F_{G})-fix(G)$ can be arbitrary large:
\begin{Lemma}\label{fffw11}
For any integer $t\geq 2$, there exists a connected graph $G$ and a
function $g$ such that $fix(F_{G})-fix(G)=t$.
\end{Lemma}

\section{The fixing number of functigraphs of some families of graphs}

In this section, we give bounds of the fixing number of functigraphs
on complete graphs, trees and join graphs. We also characterize
complete graphs for every value of $s$, where $2\leq s\leq n-2$ such
that $fix(G)= fix(F_G).$

 Following result gives the sharp upper and lower bound for
 fixing number of functigraphs of complete graphs.
\begin{Theorem}\label{f13} Let $G=K_{n}$ be the complete graph of order $n\geq
3$, and let $1<s<n$, then $$2(n-s)-1\leq fix(F_{G}) \leq 2n-s-3.$$
\end{Theorem}


\begin{proof}We assume $I= \{v_{1},v_{2},...,v_{s}\}$ and $n_{i}= |\{u\in A : v_{i}= g(u)\}|$,
for all $i$, where $1\leq i\leq s$. Also, let $j= |\{n_{i} :
n_{i}=1, 1\leq i\leq s\}|$. There are three possible cases for $j$
in functigraph $F_G$:
\begin{enumerate}
\item If $j=0$, then $2 \leq n_{i}\leq n-2$, for all $i$, where $1\leq i\leq
s.$ Thus, by definitions of $K_{n}$ and $n_{i}$, there are $s$ twin
set of vertices in $A$ and a twin-set has $n_i$ number of vertices
for each $i$, where $1\leq i\leq s$. Hence, $S^{\ast}$ contains
$\sum_{i=1}^{s}(n_{i}-1)$ vertices from $A$. Moreover, $B$ contains
$|B\setminus I|$ twin vertices and hence $S^{\ast}$ contains $n-s-1$
vertices from $B$. Hence, $|S^{\ast}|= \sum_{i=1}^{s}(n_{i}-1)+
(n-s-1) =2(n-s)-1.$
\item If $j=1,$ then without loss of generality, we assume
that $n_{s}=1$. Thus, there are $s-1$ twin-sets of vertices in $A$
and a twin-set has $n_i$ number of vertices for each $i$, where
$1\leq i\leq s-1$. Thus, $S^{\ast}$ contains
$\sum_{i=1}^{s-1}(n_{i}-1)$ vertices from $A$ and $n-s-1$ vertices
from $B$ as in the previous case. Hence,
$|S^{\ast}|=\sum_{i=1}^{s-1}(n_{i}-1)+(n-s-1) = 2(n-s)-1.$

\item If $2\leq j\leq s-1.$ Let $N=\{n_1,n_2,...,n_s\}$. We
partition the set $N$ into $N_1$ and $N_2$ where $N_1$ contains all
those elements of $N$ in which $n_i> 1$ and $N_2$ contains all those
elements of $N$ in which $n_i= 1$. Let $|N_1|=l$, then $j+l=s$ where
$j= |N_2|$. We re-index elements of $N_1$ and $N_2$ as follows:
$N_1= \{n_1^{(1)}, n_2^{(1)},..., n_l^{(1)}\}$, $N_2= \{n_1^{(2)},
n_2^{(2)},..., n_j^{(2)}\}$ where superscripts shows associations of
an element in $N_1$ or $N_2$. Now, $A$ contains $l$ twin-sets of
vertices and each twin-set has $n_k^{(1)}$ number of vertices for
each $k$, where $1\leq k\leq l$. Also, remaining $j$ vertices of $A$
are those having exactly $j$ images under $g$ and hence $j-1$ such
vertices must belong to $S^{\ast}$. For otherwise, let $u,u'\in A$
be two such vertices, then there exists an automorphism
$(uu')(g(u)g(u'))$ in $\Gamma(F_G)$. Hence, $S^{\ast}$ contains
$[\sum_{k=1}^{l}(n_{k}^{(1)}-1)]+(j-1)$ vertices from $A$. Again
$S^{\ast}$ contains $n-s-1$ vertices from $B$. Hence,
$|S^{\ast}|=[\sum_{k=1}^{l}(n_{k}^{(1)}-1)]+(j-1)+(n-s-1)=
2n-2s+j-2.$

%
%
%
%

\end{enumerate}
\end{proof}


\begin{Corollary}\label{f11} Let $G=K_{n\geq 3}$ be the complete graph and $2<s<n$ in functigraph of $G$.
If $n-s$ vertices of $A$ have same image under $g$, then
$fix(F_{G})=2n-(s+4)$.
\end{Corollary}

\begin{Corollary}\label{f11} Let $G=K_{n\geq 3}$ be the complete graph and $2<s<n$ in functigraph of $G$.
If $n-s+1$ vertices of $A$ have the same image under $g$, then
$fix(F_{G})=2n-(s+3)$.
\end{Corollary}

\begin{Corollary}\label{f12} Let $G=K_{n\geq 3}$ be the complete graph and $2<s<n$ in functigraph of $G$.
If $|g^{-1}(v)|=\frac{n}{s}$ for all $v\in I$, then
$fix(F_{G})=2(n-s)-1$.
\end{Corollary}
\begin{proof} Since, there are exactly $s$ twin sets of vertices of $A$ each of
cardinality $\frac{n}{s}$ and one twin set of $B$ of cardinality
$n-s$. Hence, $fix(F_{G})= s(\frac{n}{s}-1)+n-s-1=2(n-s)-1$.
\end{proof}

\begin{Proposition}\label{f112} For every pair of integers $n$ and $s$, where $2\leq s\leq n-2$, there are exactly
 $s-1$ complete graphs $G$ such that $fix(G)= fix(F_{G})$ for some function $g$.
\end{Proposition}
\begin{proof}We claim that for every $s$ where $2\leq
s\leq n-2$, the required $s-1$ complete graphs are $\{K_{s+i+2}:
0\leq i\leq s-2\}$ by Theorem \ref{f13}. For otherwise if $G\in
\{K_{s+i+2}: s-1\leq i\leq n-(s+2)\}$, then $fix(G=K_{s+i+2})=
s+i+1$ and by Theorem \ref{f13}, $2i+3\leq fix(F_{G={K_{s+i}}})\leq
s+2i+1.$ Since, $i>s-2$ so, $fix(G)\neq fix(F_G).$

Next, we define those functions $g_{i}: A\rightarrow B$ in
functigraph of $G$, where $G\in \{K_{s+i+2}, 0\leq i\leq s-2\}$ such
that $fix(G)= fix(F_{G})$. We discuss the following cases 
\begin{enumerate}
\item For $s=2$, we have $G=K_4$ and there are two definitions of
function $g$ satisfying hypothesis, one is defined as
$g^{-1}(v_{1})= \{u_{j}: 1\leq j\leq 3\}$, $g^{-1}(v_2)= \{u_4\}$.
 Other definition of $g$ is $g^{-1}(v_j)= \{u_{2j-1},
u_{2j}\}$ for all $j$ where $1\leq j\leq 2.$

\item For $s=3$, we have $G=K_5,K_6$. In $K_5$, $g$ is defined as
$g^{-1}(v_{1})= \{u_{j}: 1\leq j\leq 3\}$, $g^{-1}(v_2)=
\{u_4,u_5\}$. In $K_6$, again there are two definitions of $g$
satisfying hypothesis. One definition is $g^{-1}(v_{1})= \{u_{j}:
1\leq j\leq 3\}$, $g^{-1}(v_{2})= \{u_{4}, u_{5}\}$, $g(u_{6})=
v_{3}.$ Other definition of $g$ is $g^{-1}(v_{j})= \{u_{2j-1},
u_{2j}\}$ for all $j$ where $1\leq j\leq 3.$

\item For $s=4$, we have $G=K_6,K_7,K_8$. In $K_6$, $g$ is defined
as $g^{-1}(v_{1})= \{u_{j}: 1\leq j\leq 3\}$, $g^{-1}(v_j)= \{u_j\}$
for all $j$ where $4\leq j \leq 6$. In $K_7$, $g^{-1}(v_{1})=
\{u_{j}: 1\leq j\leq 3\}$, $g^{-1}(v_{2})= \{u_{j}: 4\leq j\leq
5\}$, $g(u_{j})= v_{j-3}$ for all $j$ where, $6\leq j\leq 7.$ In
$K_8$, again there are two definitions of $g$. One definition is
$g^{-1}(v_{1})= \{u_{j}: 1\leq j\leq 3\}$, $g^{-1}(v_{2})= \{u_{4},
u_{5}\}$, $g^{-1}(v_{3})= \{u_{6}, u_{7}\}$ $g(u_{8})= v_{4}.$ Other
definition of $g$ is $g^{-1}(v_{j})= \{u_{2j-1}, u_{2j}\}$ for all
$j$ where $1\leq j\leq 4.$\\
By continuing in this way.
\item For $s=n-2$, we have $G=K_n,K_{n+1},...,K_{2n-5},K_{2n-4}$.
In $K_n$, $g$ is defined as $g^{-1}(v_{1})= \{u_{j}: 1\leq j\leq
3\}$, $g^{-1}(v_j)= \{u_j\}$ for all $j$ where $4\leq j \leq n $. In
$K_{n+1}$, $g^{-1}(v_{1})= \{u_{j}: 1\leq j\leq 3\}$,
$g^{-1}(v_{2})= \{u_{j}: 4\leq j\leq 5\}$, $g(u_{j})= v_{j-3}$ for
all $j$ where, $6\leq j\leq n+1.$ In $K_{n+2}$, $g^{-1}(v_{1})=
\{u_{j}: 1\leq j\leq 3\}$, $g^{-1}(v_{2})= \{u_{j}: 4\leq j\leq
5\}$, $g^{-1}(v_{3})= \{u_{j}: 6\leq j\leq 7\}$, $g(u_{j})= v_{j-3}$
for all $j$ where, $8\leq j\leq n+2$. Continuing in similar way till
$G=K_{2n-5}$, we can find $g$ for by similar definitions. In
$K_{2n-4}$, again there are two definitions of $g$. One definition
is $g^{-1}(v_{1})= \{u_{j}: 1\leq j\leq 3\}$, $g^{-1}(v_{j})=
\{u_{2j}, u_{2j+1}\}$, where, $2\leq j\leq n-3$, $g(u_{2n-4})=
v_{n-2}.$ Other definition of $g$ is $g^{-1}(v_{j})= \{u_{2j-1},
u_{2j}\}$ for all $j$ where $1\leq j\leq n-2.$

\end{enumerate}

\end{proof}
\begin{Remark} For each $2\leq s\leq n-2$, there are exactly $s$ mapping $g: V(K_{n})\rightarrow
V(K_{n})$ such that $fix(G)= fix(F_{G}).$
\end{Remark}
\begin{Remark} $K_4$ is the only complete graph such that  $fix(G)=
fix(F_{G}),$ for all functions $g$.
\end{Remark}

Let $e^{\ast}$ be an edge of a connected graph $G$. Let $G-e^{\ast}$
is the graph obtained by deleting edge $e^{\ast}$ from graph $G.$ A
vertex $v$ of a graph $G$ is called \emph{saturated} if it is
adjacent to all other vertices of $G$.

\begin{Theorem}\label{f1122} Let $G$ be a complete graph of order $n\geq 3$ and $G_i=
G-ie^{\ast}$ for all $i$ where $1\leq i\leq\lfloor
\frac{n}{2}\rfloor$ and $e^{\ast}$ joins two saturated vertices of
the graph $G$. If $g$ is constant function, then
$$fix(F_{G_i})=\left\{
\begin{array}{ll}
2n-2i-3,\,\,\,\,\,\,\,\,\,\,\,\,\,\,\,\,\,\,\,if \,\,\,\, 1\leq i\leq \lfloor \frac{n}{2}\rfloor-1, &  \\
n-1,\,\,\,\,\,\,\,\,\,\,\,\,\,\,\,\,\,\,\,\,\,\,\,\,\,\,\,\,\,\,\,\,\, if\,\,\,\,\,i= \frac{n}{2}\,\,\,\,\,\,\,\,\,\mbox{and $n$ is even},&  \\
2\lfloor \frac{n}{2}\rfloor-1, 2\lfloor \frac{n}{2}\rfloor
\,\,\,\,\,\,\,\,\,\,\,\, if\,\,\,\,\,i=
\lfloor\frac{n}{2}\rfloor\,\,\,\,\mbox{and $n$ is odd}.
\end{array}
\right.$$
\end{Theorem}
\begin{proof}We consider the following
three cases for $i$:
\begin{enumerate}
\item For $1\leq i\leq\lfloor
\frac{n}{2}\rfloor -1$. Since, any two saturated vertices are twin
however the converse is not true. Thus, on deleting edge $e^{\ast}$
between two saturated vertices, the two vertices will no longer
remain saturated, however these will remain twin. Hence, for each
$i$ where $1\leq i\leq\lfloor \frac{n}{2}\rfloor -1$, $G_i$ contains
$n-2i$ saturated vertices and $i$ twin set of vertices each of
cardinality two. Hence $fix(G_i)= n-i-1.$ Now, if $g$ is constant,
then $|S^{\ast}|= (n-i-1)+(n-i-2)=2n-2i-3.$
\item If $i=\frac{n}{2}$ and $n$ is even. Then $G_i$ contains
no saturated vertex and $\frac{n}{2}$ twin set of vertices each of
cardinality two. Hence $fix(G_i)= \frac{n}{2}.$ Now, if $g$ is
constant, then $|S^{\ast}|= \frac{n}{2}+(\frac{n}{2}-1)=n-1.$
\item If $i=\lfloor \frac{n}{2}\rfloor$ and $n$ is odd. Then $G_i$ contains
one saturated vertex and $\lfloor\frac{n}{2}\rfloor$ twin set of
vertices each of cardinality two. Hence $fix(G_i)=
\lfloor\frac{n}{2}\rfloor.$ Suppose that, $u'= g(u_{i})$ for all
$u_i\in V(G_i)$ where $1\leq i\leq n$. Now, if $u'$ is a twin vertex
then $|S^{\ast}|=
\lfloor\frac{n}{2}\rfloor+(\lfloor\frac{n}{2}\rfloor-1)=2\lfloor
\frac{n}{2}\rfloor -1.$ However, if $u'$ ia saturated vertex, then
$|S^{\ast}|=
\lfloor\frac{n}{2}\rfloor+\lfloor\frac{n}{2}\rfloor=2\lfloor
\frac{n}{2}\rfloor.$
\end{enumerate}
\end{proof}


From Theorem \ref{f1122}, we can establishes the sharp bounds for
the fixing number of a functigraph of $G_i= G-ie^{\ast}$ in the
following corollary.
\begin{Corollary}
Let $G$ be a complete graph of order $n\geq 3$ and $G_i=
G-ie^{\ast}$ for all $i$ where $1\leq i\leq\lfloor
\frac{n}{2}\rfloor$ and $e^{\ast}$ joins two saturated vertices of
the graph $G$. If $g$ is a constant function, then $fix(G)\leq
fix(F_{G_i}) \leq fix (F_G).$ Both bounds are sharp.
\end{Corollary}

Let $T$ be a tree graph and $v\in V(T)$. If $deg_{T}(v)=1$, then $v$
is called a \emph{pendant vertex}. A vertex $v\in T$ that adjacent
to a pendant vertex is called a \emph{support vertex}. We denote the
total number of pendant vertices in a tree $T$ by $p(T)$. We denote
the total number of support vertices in a tree $T$ by $s(T)$.



%

\begin{Proposition}\label{proppw11} If $T$ is a symmetric tree of order $n\geq
2$, then $fix(F_{T})\leq 2fix(T).$ This bound is sharp.
\end{Proposition}

\begin{Corollary}If $P_{n}$ is a path of order $n\geq
2$, then $fix(F_{P_{n}})\leq 2.$ This bound is sharp.
\end{Corollary}

\begin{Theorem}\label{thw11m} Let $T$ be a symmetric tree and $F_{T}$ be its symmetric functigraph, then
 $fix(F_{T})= 2|T|-t,\,\ 2\leq t\leq 3$ if and only if $T=P_{2}.$
\end{Theorem}
\begin{proof}If $T=P_{2}$, then $F_{T}$ is either $C_{4}$ or $K_{3}$ with a pendant vertex. Hence, $fix(F_{T})=
2|T|-t,\,\ 2\leq t\leq 3$. Conversely, suppose that, $fix(F_{T})=
2|T|-t,\,\ 2\leq t\leq 3$. We discuss the following cases for
$s(T).$
\begin{enumerate}
\item If $s(T)\geq 3$. We partition $V(T)$ into the sets $X_{1}$, $X_{2}$ and
$X_{3}$ where $X_{1}= \{u\in V(T)$ : $u$ is a pendant vertex of
$T\},$ $X_{2}= \{u\in V(T)$ : $u$ is a support vertex of $T\}$ and
$X_{3}= V(T)\setminus \{X_{1}\cup X_{2}\}.$ Let $X_{4}= X_{2}\cup
X_{3},$ then $fix(T)\leq |T|-|X_{4}|-1$ and from Proposition
\ref{proppw11}, $fix(F_{T})\leq 2[|T|-|X_{4}|-1]$ which leads to a
contradiction as $|X_{4}|\geq 3$.
\item If $s(T)=2$ then, we have the following
two subcases:
\begin{enumerate}
\item If $s(T)=p(T)$, then $T= P_{n\geq 2} (n\neq 3)$ and hence
$fix(F_{P_{n}})\leq 2$. Thus, by hypothesis, $T=P_{2}.$

\item If $s(T)\neq p(T)$, then either $p(T)=1$ or $p(T)>
2$. However, $p(T)=1$ and $s(T)=2$ is impossible in a tree, so
$p(T)> 2$. This also leads to a contradiction as in Case (1).
\end{enumerate}
\item If $s(T)=1$, then $T= K_{1,n}, n\geq 2$ and hence $fix(F_{T})\leq
2[|T|-2]$ which is again a contradiction.
\end{enumerate}
\end{proof}
The following corollary can be proved by using similar arguments as
in proof of Theorem \ref{thw11m}.
\begin{Corollary} Let $T$ be a symmetric tree and $F_{T}$ be its symmetric
functigraph. (1) If $fix(F_{T})= 2|T|-t,\,\ 4\leq t\leq 5$, then $T=K_{1,n}, n\geq 2.$\\
(2) If $fix(F_{T})= 2|T|-6$, then $T=K_{1,n}, n\geq 3.$\\
(3) If $fix(F_{T})= 2|T|-7$, then $T\in \{P_{4}, K_{1,n}, n\geq 3\}.$\\
(4) If $fix(F_{T})= 2|T|-8$, then $T\in \{P_{5},\ K_{1,n},\ n\geq 4,
(K_{1,n_{1}},K_{1,n_{2}}) +e, n_{1}, n_{2}\geq 2, K_{1,n}\ \mbox{and
a vertex adjacent with one pendant vertex of}\,\ K_{1,n}\}.$
\end{Corollary}

Suppose that $G_{1}=(V_{1},E_{1})$ and $G_{2}=(V_{2},E_{2})$ be two
graphs with disjoint vertex sets $V_{1}$ and $V_{2}$ and disjoint
edge sets $E_{1}$ and $E_{2}$. The \emph{union} of $G_{1}$ and
$G_{2}$ is the graph $G_{1}\cup G_{2}=(V_{1}\cup V_{2}, E_{1}\cup
E_{2})$. The \emph{join} of $G_{1}$ and $G_{2}$ is the graph
$G_{1}+G_{2}$ that consists of $G_{1}\cup G_{2}$ and all edges
joining all vertices of $V_{1}$ with all vertices of $V_{2}$.

\begin{Theorem} \label{fazil123} \cite{fazil} Let $G_{1}$ and $G_{2}$ be two connected
graphs, then $fix(G_{1}+G_{2})\geq fix(G_{1})+fix(G_{2}).$ This
bound is sharp.
\end{Theorem}

\begin{Proposition}\label{fazil22}  Let $G_{1}$ and $G_{2}$ be two connected
graphs and $g: V(G_{1}+G_{2})\rightarrow V(G_{1}+G_{2})$ be a
constant function, then $$fix(F_{G_{1}+G_{2}})=
2fix(G_{1}+G_{2})-i,\,\ 0\leq i\leq 1.$$
\end{Proposition}
\begin{proof} Let $A$ and $B$ be two copies of $G_{1}+G_{2}$. Let $S_{1}$ and
$S_{2}$ be minimum fixing sets of $A$ and $B$, respectively. Define
$g:A\rightarrow B$ by $u'= g(u)$, for all $u\in A$. We discuss the
following two cases.

\begin{enumerate}
\item If $G_{1}$ does not contain any saturated vertex,
then $G_{1}+G_{2}$ also does not contain any saturated vertex. In
this case $fix(A)= fix(B)= fix(G_{1})+fix(G_{2})$ by Theorem
\ref{fazil123}. Now, if $u'\in S_{2},$ then $S^{\ast}= S_{1}\cup
\{S_{2}\setminus \{u'\}\}$ because $\theta(u')= \{u'\}$ in
$F_{G_{1}+G_{2}}.$ If $u'\notin S_{2}$, then $S^{\ast}= S_{1}\cup
S_{2}.$
\item If $G_{1}$ contains any saturated vertex, then both $A$ and
$B$ have saturated vertices. In this case $fix(A) = fix(B)>
fix(G_{1})+fix(G_{2})$ by Theorem \ref{fazil123}. Now, if $u'$ is a
saturated vertex of $B$, then $u'\in S_{2}$ and hence $S^{\ast}=
S_{1}\cup \{S_{2}\setminus \{u'\}\}.$ If $u'$ is not a saturated
vertex, then either $u'\in S_{2}$ or $u'\notin S_{2}.$ If $u'\in
S_{2}$, then $S^{\ast}= S_{1}\cup \{S_{2}\setminus \{u'\}\}$ and if
$u'\notin S_{2},$ then $S^{\ast}= S_{1}\cup S_{2}.$
\end{enumerate}

\end{proof}

\end{document}